\documentclass[smallextended]{svjour3}       
\smartqed  
\usepackage{graphicx}
\usepackage{color}
\usepackage{color}

\definecolor{midgrey}{RGB}{150,173,180}
\usepackage{verbatim}
\usepackage{amsfonts}
\usepackage{amssymb}
\usepackage{yhmath}

\usepackage{mathtools}

\newcommand{\komp}      {{^\prime}}

\newcommand{\kompM}[1]{^{{\prime^{\mkern-4mu^{_{#1}}}}}}
\newcommand{\nega}      [1] {{#1}\komp}
\newcommand{\negaM}      [2] {{#2}\kompM{#1}}

\newcommand{\te}{{\mathbin{*\mkern-9mu \circ}}}

\newcommand{\ite}[1]{\mathbin{\rightarrow_{#1}}}

\newcommand{\g}                 [2] {{#1} \mathbin{\te} {#2}}

\newcommand{\res}               [3] {{#2}\mathbin{\ite{#1}}{#3}}

\newcommand{\lex}                                        {\overset{\leftarrow}{\times}}

\newcommand{\Twoheadleftarrow}               {\leftarrow \mkern-12.25mu \leftharpoonup}
\newcommand{\threeheadleftarrow}               {\leftarrow \mkern-12.5mu \leftarrow}

\newcommand{\plexII}               {\overset{\Twoheadleftarrow}{\times}}
\newcommand{\plexI}               {\overset{\threeheadleftarrow}{\times}}

\newcommand{\PLPI}            [3] {{#1}_{{#2}}\plexI{{#3}}}
\newcommand{\PLPII}            [2] {{#1}\plexII{{#2}}}

\newcommand{\PLPIII}            [4] {{#1}_{{#2}_{#3}}\plexI{{#4}}}
\newcommand{\PLPIV}            [3] {{#1}_{#2}\plexII{{#3}}}

\begin{document}


\title{Involutive uninorm logic with fixed point enjoys finite strong standard completeness\thanks{The present scientific contribution was supported by the GINOP 2.3.2-15-2016-00022 grant
and the Higher Education Institutional Excellence Programme 20765-3/2018/FEKUTSTRAT of the Ministry of Human Capacities in Hungary.}
}


\titlerunning{$\mathbf{IUL}^{fp}$ enjoys finite strong standard completeness}        

\author{S\'andor Jenei}


\institute{S\'andor Jenei \at
              Institute of Mathematics and Informatics, University of P\'ecs, Ifj\'us\'ag u. 6, H-7624 P\'ecs, Hungary \\
              Tel.: +36-72-503600,24684\\
              Fax: +36-72-503697\\
              \email{jenei@ttk.pte.hu}
              }

\date{}

\maketitle

\begin{abstract}
An algebraic proof is presented for the finite strong standard completeness of  involutive uninorm logic with fixed point.
The result may provide a first step towards settling the open standard completeness problem for involutive uninorm logic posed in \cite{MetMon 2007}.

\end{abstract}

\section{Introduction and preliminaries}

Mathematical fuzzy logics have been introduced in \cite{hajekbook}, and the topic is a rapidly growing field ever since (\cite{MetETAL,Handbooks}).
Substructural fuzzy logics were introduced in \cite{MetMon 2007} as substructural logics that are standard complete, that is, complete with respect to algebras whose lattice reduct is the real unit interval $[0,1]$, and standard completeness for several substructural logics, all are stronger than the there-introduced Uninorm Logic ($\mathbf{UL}$), has been proven, too, with the notable exception of the Involutive Uninorm Logic ($\mathbf{IUL}$).
Its standard completeness has remained an open problem, which has stood against the attempts using the widely used embedding method of \cite{JMstcompl} or the density elimination technique of \cite{MetMon 2007}.
An algebraic semantics for $\mathbf{IUL}$ is the variety of bounded involutive FL$_e$-chains. 
The structure of involutive FL$_e$-chains is quite rich; even the integral algebras of the class, the class of IMTL-chains, is a class containing, e.g., the connected rotations \cite{rotinvsemigroups} of all MTL-chains.
This richness renders their algebraic description a hard task.
Therefore, as a first step toward this aim, we focus on odd FL$_e$-chains, which is a subclass of involutive FL$_e$-chains.
Bounded odd FL$_e$-chains constitute the algebraic semantics for the Involutive Uninorm Logic with Fixed Point ($\mathbf{IUL}^{fp}$) \cite{GMPhD}.
Prominent examples of odd FL$_e$-algebras are lattice-ordered abelian groups and odd Sugihara algebras.
The former constitutes an algebraic semantics for Abelian Logic \cite{MeyerAbelian,cesari,AbelianNow} while the latter constitutes an algebraic semantics for $\mathbf{IUML}^*$, which is a logic at the intersection of relevance logic and many-valued logic \cite{GalRaf}.
\\
\indent 
In order to make a possible first step toward the settling the standard completeness problem for $\mathbf{IUL}$, in this paper we prove the finite strong standard completeness of $\mathbf{IUL}^{fp}$, a somewhat simpler logic than $\mathbf{IUL}$.
The key ingredient in our proof is an embedding, which is based on a representation theorem for the class of odd FL$_e$-chains which possess only finitely many positive idempotent elements, by means of linearly ordered abelian groups and a modified version of the lexicographic product construction \cite{Jenei_Hahn}.

\medskip
An {\em FL$_e$-algebra} 
is a structure $( X, \wedge,\vee, \te, \ite{\te}, t, f )$ such that 
$( X, \wedge,\vee )$ is a lattice, $( X,\leq, \te,t)$ is a commutative, 
residuated\footnote{That is, there exists a binary operation $\ite{\te}$,
called the residual operation of $\te$, such that $\g{x}{y}\leq z$ if and only if $\res{\te}{x}{z}\geq y$; this equivalence is called residuation condition or adjointness condition, ($\te,\ite{\te}$) is called an adjoint pair. Equivalently, for any $x,z$, the set $\{v\ | \ \g{x}{v}\leq z\}$ has its greatest element, and $\res{\te}{x}{z}$, the residuum of $x$ and $z$, is defined as this element: $\res{\te}{x}{z}:=\max\{v\ | \ \g{x}{v}\leq z\}$.}
monoid, 
and $f$ is an arbitrary constant. One defines $\nega{x}=\res{\te}{x}{f}$ and calls an FL$_e$-algebra {\em involutive} if $\nega{(\nega{x})}=x$ holds.
An FL$_e$-algebra is called {\em odd} if it is involutive and $t=f$. 
Algebras will be denoted by bold capital letters, their underlying likewise by the same regular letter.
Commutative residuated lattices are exactly the $f$-free reducts of FL$_e$-algebras.

For an FL$_e$-algebra $\mathbf X=( X, \wedge, \vee, \te, \ite{\te}, t, f )$, denote $\tau(x)=\res{\te}{x}{x}$ for $x\in X$.
Let $X_{gr}=\{ x\in X \ | \g{x}{\nega{x}}=t \}$. 
If $\mathbf X$ is odd then there is a subalgebra $\mathbf X_\mathbf{gr}$ of $\mathbf X$ over $X_{gr}$,
call it the group part of $\mathbf X$.

\begin{lemma}\label{tau_lemma}{\rm \cite{Jenei_Hahn}}
Let  $\mathbf X=( X,  \wedge, \vee, \te, \ite{\te}, t, f )$ be an involutive FL$_e$-algebra.
The following statements hold true.
\begin{enumerate}
\item\label{id_RangeT}
$\{ \tau(x): x\in X \}$ is equal to the set of positive idempotent elements of $X$. 
\item\label{tauINHERITED}
If $\mathbf X$ is an odd FL$_e$-chain then the $\tau$ value of any expression, which contains only the operations $\te$, $\ite{\te}$ and \, $\komp$ equals the maximum of the $\tau$-values of its variables and constants.\qed
\end{enumerate}
\end{lemma}
Let $(X, \leq)$ be a 
poset.
For $x\in X$ define 
$$
x_\downarrow=
\left\{
\begin{array}{ll}
z & \mbox{if there exists a unique $z\in X$ such that $x$ covers $z$,}\\
x & \mbox{otherwise.}\\
\end{array}
\right.
$$
Define $x_\uparrow$ dually. 
We say for $Z\subseteq X$ that $Z$ is {\em discretely embedded} into $X$ if 
for $x\in Z$ it holds true that $x\notin\{ x_\uparrow,x_\downarrow\}\subseteq Z$. 
Denote the lexicographic product by $\lex$.

\begin{definition}\label{FoKonstrukcio}{\rm \cite{Jenei_Hahn}}
\rm
Let ${\mathbf X}=(X, \wedge_X,\vee_X, \ast, \ite{\ast}, t_X, f_X)$ be an odd FL$_e$-algebra
and ${\mathbf Y}=( Y, \wedge_Y,\vee_Y, \star, \ite{\star}, t_Y, f_Y )$
be an involutive FL$_e$-algebra, with residual complement $\kompM{\ast}$ and $\kompM{\star}$, respectively.

\medskip
\begin{enumerate}
\item[A.]
Add a new element $\top$ to $Y$ as a top element and annihilator (for $\star$), then add a new element $\bot$ to $Y\cup\{\top\}$ as a bottom element and annihilator.
Extend $\kompM{\star}$ by $\negaM{\star}{\bot}=\top$ and $\negaM{\star}{\top}=\bot$.
Let 
$\mathbf V\leq\mathbf Z\leq\mathbf X_\mathbf{gr}$. 
Let 
$$
\PLPIII{X}{Z}{V}{Y}= (V\times (Y\cup\{\top,\bot\}))\cup ((Z\setminus V)\times\{\top,\bot\})\cup \left((X\setminus Z)\times \{\bot\}\right),
$$
and let $\PLPIII{\mathbf X}{\mathbf Z}{\mathbf V}{\mathbf Y}$, the {\em type III partial lexicographic product} of $\mathbf X,\mathbf Z,\mathbf V$ and $\mathbf Y$ be given by
$$\PLPIII{\mathbf X}{\mathbf Z}{\mathbf V}{\mathbf Y}=\left(\PLPIII{X}{Z}{V}{Y}, \leq, \te, \ite{\te}, (t_X,t_Y),(f_X,f_Y)\right),$$
where $\leq$ is the restriction of the lexicographical order of $\leq_X$ and $\leq_{Y\cup\{\top,\bot\}}$ to $\PLPIII{X}{Z}{V}{Y}$, 
$\te$ is defined coordinatewise, and the operation $\ite{\te}$ is given by
$
\res{\te}{(x_1,y_1)}{(x_2,y_2)}=\nega{\left(\g{(x_1,y_1)}{\nega{(x_2,y_2)}}\right)} ,
$
where
$$
\nega{(x,y)}=\left\{
\begin{array}{ll}
(\negaM{\ast}{x},\bot) 		& \mbox{if $x\not\in Z$}\\
(\negaM{\ast}{x},\negaM{\star}{y}) 	& \mbox{if $x\in Z$}\\
\end{array}
\right. .
$$
In the particular case when $\mathbf V=\mathbf Z$, we use the simpler notation 
$\PLPI{\mathbf X}{\mathbf Z}{\mathbf Y}$ for $\PLPIII{\mathbf X}{\mathbf Z}{\mathbf V}{\mathbf Y}$
and call it
the {\em type I partial lexicographic product} 
of $\mathbf X,\mathbf Z$, and $\mathbf Y$.
\item[B.]\label{B}
Assume that $X_{gr}$ 
is discretely embedded into $X$.
Add a new element $\top$ to $Y$ as a top element and annihilator.
Let 
$\mathbf V\leq\mathbf X_\mathbf{gr}$. 
Let
$$
\PLPIV{X}{V}{Y}=(X\times \{\top\})\cup (V\times Y) 
$$
and let 
$\PLPIV{\mathbf X}{\mathbf V}{\mathbf Y}$, 
the {\em type IV partial lexicographic product} of $\mathbf X$, $\mathbf V$ and $\mathbf Y$ be given by
$$\PLPIV{\mathbf X}{\mathbf V}{\mathbf Y}=\left(\PLPIV{X}{V}{Y}, \leq, \te, \ite{\te}, (t_X,t_Y),(f_X,f_Y)\right),$$
where $\leq$ is the restriction of the lexicographical order of $\leq_X$ and $\leq_{ Y\cup\{\top\}}$ to 
$\PLPIV{X}{V}{Y}$,
$\te$ is defined coordinatewise, and the operation $\ite{\te}$ is given by
$
\res{\te}{(x_1,y_1)}{(x_2,y_2)}=\nega{\left(\g{(x_1,y_1)}{\nega{(x_2,y_2)}}\right)} ,
$
where $\komp$ is defined by
\begin{equation}\label{FuraNegaEXT}
\nega{(x,y)}=
\left\{
\begin{array}{ll}
(\negaM{\ast}{x},\top) 			& \mbox{if $x\not\in X_{gr}$ and $y=\top$}\\
((\negaM{\ast}{x})_\downarrow,\top) 	& \mbox{if $x\in X_{gr}$ and $y=\top$}\\
(\negaM{\ast}{x},\negaM{\star}{y}) 	& \mbox{if $x\in V$ and $y\in Y$}\\
\end{array}
\right. .
\end{equation}
In the particular case when $\mathbf V=\mathbf X_{\mathbf{gr}}$, we use the simpler notation 
$\PLPII{\mathbf X}{\mathbf Y}$ for $\PLPIV{\mathbf X}{\mathbf V}{\mathbf Y}$
and call it
the {\em type II partial lexicographic product} 
of $\mathbf X$ and $\mathbf Y$.
\end{enumerate}
\end{definition}

\begin{lemma}\label{SubLexiTheo}{\rm \cite{Jenei_Hahn}}
Adapt the notation of Definition~\ref{FoKonstrukcio}.
$\PLPIII{\mathbf X}{\mathbf Z}{\mathbf V}{\mathbf Y}$
and
$\PLPIV{\mathbf X}{\mathbf V}{\mathbf Y}$
are
involutive FL$_e$-algebras with the same rank\footnote{The rank of an involutive FL$_e$-algebra is positive if $t>f$, negative if $t<f$, and $0$ if $t=f$.} as that of $\mathbf Y$.
In particular, if $\mathbf Y$ is odd then so are 
$\PLPIII{\mathbf X}{\mathbf Z}{\mathbf V}{\mathbf Y}$
and
$\PLPIV{\mathbf X}{\mathbf V}{\mathbf Y}$.
In addition, 
$\PLPIII{\mathbf X}{\mathbf Z}{\mathbf V}{\mathbf Y}\leq\PLPI{\mathbf X}{\mathbf Z}{\mathbf Y}$ and
$\PLPIV{\mathbf X}{\mathbf V}{\mathbf Y}\leq\PLPII{\mathbf X}{\mathbf Y}$.
\end{lemma}

The following theorem asserts that up to isomorphism, any odd FL$_e$-chain which has only finitely many positive idempotent elements can be built by iterating finitely many times the type III and type IV partial lexicographic product constructions using only linearly ordered abelian groups, as building blocks.
\begin{theorem}\label{Hahn_type}{\rm \cite{Jenei_Hahn}}
{{\bf (Group representation)}}
If $\mathbf X$ is an odd FL$_e$-chain, which has only $n\in\mathbb N$, $n\geq 1$ positive idempotent elements then there exist linearly ordered abelian groups $\mathbf G_i$ $(i\in\{1,2,\ldots,n\})$, 
$\mathbf V_1\leq\mathbf Z_1\leq \mathbf G_1$, $\mathbf V_i\leq\mathbf Z_i\leq\mathbf V_{i-1}\lex\mathbf G_i$ $(i\in\{2,\ldots,n-1\})$,
and a binary sequence $\iota\in \{III,IV\}^{\{2,\ldots,n\}}$
such that 
$\mathbf X\simeq\mathbf X_n$, where
$\mathbf X_1:=\mathbf G_1$ and for $i\in\{2,\ldots,n\}$,
\begin{equation}\label{EzABeszed}
\mathbf X_i:=
\left\{
\begin{array}{ll}
\PLPIII{\mathbf X_{i-1}}{\mathbf Z_{i-1}}{\mathbf V_{i-1}}{\mathbf G_i}	& \mbox{ if $\iota_i=III$}\\
\PLPIV{\mathbf X_{i-1}}{\mathbf V_{i-1}}{\mathbf G_i} 	& \mbox{ if $\iota_i=IV$}\\
\end{array}
\right. .
\end{equation}
\end{theorem}

\begin{remark}
Denote the one-element odd FL$_e$-algebra by $\mathbf 1$.
For any bounded odd FL$_e$-algebra $\mathbf X$, in its group representation $\mathbf G_1=\mathbf 1$, since all other linearly ordered groups are infinite and unbounded, and 
both type III and type IV constructions preserve boundedness of the first component.
\end{remark}

A linearly ordered set $(X,\leq)$ is called Dedekind complete if every non-empty subset of $X$ bounded from above has a supremum.

\begin{lemma}{\rm (\cite[Theorem 2.29]{real}}\label{realisLEMMA}
A linearly ordered set $(K,\leq)$ is order isomorphic to
$\mathbb R$ if and only if
$(K,\leq)$ possesses the following four properties:
$(K,\leq)$ has no least neither greatest element,
$(K,\leq)$ is densely ordered,  there exists a countable dense subset of  $(K,\leq)$, and $(K,\leq)$ is Dedekind complete.
\end{lemma}

For an introduction to the basic notions and basic results in the field of mathematical fuzzy logics, see \cite{introFUZZY}.
For the definition of $\mathbf{IUL}$ the reader is referred to \cite{GMPhD}.
The logic $\mathbf{IUL}^{fp}$ is defined as $\mathbf{IUL}$ extended by the following axiom: $\mathbf t \leftrightarrow\mathbf f$.
A mathematical fuzzy logic $L$ enjoys finite strong standard completeness if the following conditions are equivalent for each formula $\varphi$ and each finite theory $T$: 
(i) $T\vDash_L \varphi$. (ii) For each standard L-algebra\footnote{(A standard L-algebra is an L-algebra over the real unit interval $[0,1]$)} $\mathbf A$ and each $\mathbf A$-model $e$ of $T$, $e$ is an $\mathbf A$-model of $\varphi$.
A possible way of proving finite strong standard completeness is to embed finitely generated L-algebras into standard L-algebras, which we shall do in Theorem~\ref{BeAgyazatam}.

\section{$\mathbf{IUL}^{fp}$ enjoys finite strong standard completeness}\label{222}


To show that Theorem~\ref{Hahn_type} is applicable to our case, first we prove that each finitely generated odd FL$_e$-chain possesses only finitely many positive idempotents.

\begin{proposition}\label{fingen_finIdemp}
Any finitely generated odd FL$_e$-chain has only finitely many positive idempotent elements.
\end{proposition}
\begin{proof}
Since the order is total, $\{ \tau(x\wedge y),\tau(x\vee y)\}\in\{\tau(x), \tau(y)\}$ holds.
Therefore, using claim~\ref{tauINHERITED} in Lemma~\ref{tau_lemma}, 
an easy induction on the recursive structure of the formula that generates a given element $x$ of the algebra shows that
the $\tau$-value of $x$ coincides with the $\tau$-value of one of its generators or constants.
Therefore, all elements of the algebra share only finitely many $\tau$-values, since the algebra is finitely generated.
Claim~\ref{id_RangeT} in Lemma~\ref{tau_lemma} concludes the proof.
\qed\end{proof}

\begin{proposition}\label{typeIIisVAN}
Any odd FL$_e$-chain which has finitely many positive idempotent elements is embeddable into an odd FL$_e$-chain,
which has the same number of positive idempotent elements, and
which has a type I-II group representation.
\end{proposition}
\begin{proof}
Let $\mathbf Y$ be an odd FL$_e$-chain which has finitely many positive idempotent elements, and let a type III-IV group representation of $\mathbf Y$ be given, according to Theorem~\ref{Hahn_type}, 
that is,
$\mathbf G_i$ $(i\in\{1,2,\ldots,n\})$ linearly ordered abelian groups, 
$\mathbf V_{1}\leq\mathbf Z_{1}\leq \mathbf G_1$ and 
$\mathbf V_{i}\leq\mathbf Z_{i}\leq\mathbf V_{i-1}\lex\mathbf G_i$
($(i\in\{2,\ldots,n-1\})$),
and $\iota\in \{III,IV\}^{\{2,\ldots,n\}}$ such that 

(i) $\mathbf Y\simeq\mathbf Y_n$, where
\begin {equation}\label{EzazY}
\mathbf Y_1=\mathbf G_1 \mbox{\ and \ } 
\mathbf Y_i=
\left\{
\begin{array}{ll}
\PLPIII{\mathbf Y_{i-1}}{\mathbf Z_{i-1}}{\mathbf V_{i-1}}{\mathbf G_i} & \mbox{if $i\in\{2,\ldots,n\}$ and $\iota_i=III$}\\
\PLPIV{\mathbf Y_{i-1}}{\mathbf V_{i-1}}{\mathbf G_i} & \mbox{if $i\in\{2,\ldots,n\}$ and $\iota_i=IV$}\\
\end{array}
\right.
\end{equation}

(ii) 
for $i\in\{2,\ldots,n\}$,
if $\iota_i=IV$ then 
$\mathbf Z_{i-1}=(\mathbf Y_{i-1})_\mathbf{gr}$ and 
$(\mathbf Y_{i-1})_\mathbf{gr}$ is discretely embedded into $\mathbf Y_{i-1}$.

\noindent
We define   
\begin{equation}\label{ezLESZazX}
\mathbf X_1=\mathbf G_1 \mbox{\ and \ } 
\mathbf X_i=
\left\{
\begin{array}{ll}
\PLPI{\mathbf X_{i-1}}{\mathbf Z_{i-1}}{\mathbf G_i} & \mbox{if $i\in\{2,\ldots,n\}$ and $\iota_i=III$}\\
\PLPII{\mathbf X_{i-1}}{\mathbf G_i} & \mbox{if $i\in\{2,\ldots,n\}$ and $\iota_i=IV$}\\
\end{array}
\right. .
\end{equation}
To see that $\mathbf X_i$ is well-defined, we need to verify, according to Definition~\ref{FoKonstrukcio}, that
\begin{itemize}
\item for $i\in\{2,\ldots,n-1\}$,
$\mathbf Z_{i-1}\leq(\mathbf X_{i-1})_\mathbf{gr}$ holds.

Indeed, 
for $i=2$,
$(\mathbf X_1)_\mathbf{gr}$ is equal to $\mathbf G_1$ and $\mathbf Z_{1}\leq \mathbf G_1$ holds by assumption.
For $i>2$, by assumption, $\mathbf Z_{i-1}\leq\mathbf V_{i-2}\lex\mathbf G_{i-1}$ which is equal to $(\mathbf Y_{i-1})_\mathbf{gr}$. Also by assumption, the latest is a subalgebra of $\mathbf Z_{i-2}\lex\mathbf G_{i-1}$ which is equal to $(\mathbf X_{i-1})_\mathbf{gr}$.

\item 
In addition, we need to verify that 
for $i\in\{2,\ldots,n\}$, if $\iota_i=IV$ then 
$(\mathbf X_{i-1})_\mathbf{gr}$ is discretely embedded into $\mathbf X_{i-1}$.

If $i=2$ then it holds by assumption that 
$(\mathbf X_1)_\mathbf{gr}=(\mathbf Y_1)_\mathbf{gr}$ is discretely embedded into
$\mathbf Y_1=\mathbf X_1$.

For $i>2$, by assumption it holds that 
$(\mathbf Y_{i-1})_\mathbf{gr}=\mathbf V_{i-2}\lex\mathbf G_{i-1}$ is discretely embedded into $\mathbf Y_{i-1}$.
That is, 
\begin{equation}\label{LexiDiscrEmb}
\mbox{
for $x\in\mathbf V_{i-2}\lex\mathbf G_{i-1}$ it holds true that
$x\notin\{ x_\uparrow,x_\downarrow\}\subseteq \mathbf V_{i-2}\lex\mathbf G_{i-1}$,
} 
\end{equation}
where $x_\downarrow$ and $x_\uparrow$ are computed in $\mathbf Y_{i-1}$.
It cannot be that case that $\mathbf G_{i-1}$ is trivial, since then for $h\in\mathbf V_{i-2}$, $(h,1)_\uparrow=(h,\top)$ would not be an element of $\mathbf V_{i-2}\lex\mathbf G_{i-1}$.
Thus $\mathbf G_{i-1}$ is unbounded.
Therefore, because of the lexicographic ordering on $Y_{i-1}$, (\ref{LexiDiscrEmb}) holds if and only if  $\mathbf G_{i-1}$ is discretely ordered.
Since $\mathbf G_{i-1}$ is discretely ordered, it follows that
$(\mathbf X_{i-1})_\mathbf{gr}$ which is equal to either
$\mathbf Z_{i-2}\lex\mathbf G_{i-1}$
or
$(\mathbf X_{i-2})_\mathbf{gr}\lex\mathbf G_{i-1}$ is discretely embedded into $\mathbf X_{i-1}$, as required.
\end{itemize}
We claim that for $i\in\{1,\ldots,n\}$, $\mathbf Y_i \leq \mathbf X_i$ holds.
Indeed, the claim being straightforward for $i=1$, assume $\mathbf Y_{i-1} \leq \mathbf X_{i-1}$ for $i\in\{2,\ldots,n\}$.
By Lemma~\ref{SubLexiTheo}, 
$\mathbf Y_i=\PLPIII{\mathbf Y_{i-1}}{\mathbf Z_{i-1}}{\mathbf V_{i-1}}{\mathbf G_i}\leq\PLPI{\mathbf Y_{i-1}}{\mathbf Z_{i-1}}{\mathbf G_i}$
(if $\iota_i=III$) or
$\mathbf Y_i=\PLPIV{\mathbf Y_{i-1}}{\mathbf V_{i-1}}{\mathbf G_i}\leq\PLPII{\mathbf Y_{i-1}}{\mathbf G_i}$ (if $\iota_i=IV$).
Using $\mathbf Z_{i-1}\leq(\mathbf X_{i-1})_\mathbf{gr}$ (see above),
together with $\mathbf Y_{i-1}\leq\mathbf X_{i-1}$ (see the induction hypothesis),
it follows that
$
\PLPI{\mathbf Y_{i-1}}{\mathbf Z_{i-1}}{\mathbf G_i}
$
is a subalgebra of
$
\PLPI{\mathbf X_{i-1}}{\mathbf Z_{i-1}}{\mathbf G_i}
=\mathbf X_i
$
(if $\iota_i=III$)
or
$
\PLPII{\mathbf Y_{i-1}}{\mathbf G_i}
$
is a subalgebra of
$
\PLPII{\mathbf X_{i-1}}{\mathbf G_i}
=\mathbf X_i
$
(if $\iota_i=IV$).
Thus we have shown 
$$\mathbf Y\simeq\mathbf Y_n\leq\mathbf X_n,$$ as required.
\qed\end{proof}

\begin{proposition}\label{aGROUPREPRisFINITELYgeneratedTAGOKBOLall}
If $\mathbf Y$ is a finitely generated odd FL$_e$-chain then all the groups in its group representation are finitely generated, too.
\end{proposition}
\begin{proof}
Let $\mathbf Y$ be a finitely generated odd FL$_e$-algebra, and consider its type III-IV group representation, as in the proof of Proposition~\ref{typeIIisVAN}.
We start by observing that 
$
\mathbf Y_n
$
is finitely generated, since so is
$
\mathbf Y
$.
Since
$$
\mathbf Y_n=
\left\{
\begin{array}{ll}
\PLPIII{\mathbf Y_{n-1}}{\mathbf Z_{n-1}}{\mathbf V_{n-1}}{\mathbf G_n} & \mbox{if $\iota_n=III$}\\
\PLPIV{\mathbf Y_{n-1}}{\mathbf V_{n-1}}{\mathbf G_n} & \mbox{if $\iota_n=IV$}\\
\end{array}
\right.
$$
and 
$Y_{n-1}\times\{\top\}$ is a subset of the universe of $\mathbf Y_n$ in both cases,
it follows that 
$
\mathbf Y_{n-1}
$
is finitely generated, too, so an easy downward induction shows that 
$
\mathbf Y_i
$
is finitely generated for $i\in\{n,n-1,\ldots,1\}$.
In particular, $\mathbf G_1=\mathbf Y_1$ is is finitely generated.
Since 
$$
\mathbf Y_2=
\left\{
\begin{array}{ll}
\PLPIII{\mathbf Y_1}{\mathbf Z_1}{\mathbf V_1}{\mathbf G_2} & \mbox{if $\iota_2=III$}\\
\PLPIV{\mathbf Y_1}{\mathbf V_1}{\mathbf G_2} & \mbox{if $\iota_2=IV$}\\
\end{array}
\right. ,
$$
and
$
\mathbf Y_2
$
is finitely generated, 
and $\{t\}\times G_2$ and $Z_1\times\{t_{G_2}\}$ are subsets of the universe of $\mathbf Y_2$,
it follows that
$
\mathbf G_2
$
and
$
\mathbf Z_{1}
$
are finitely generated, and so on. 
So an (upward) induction shows that 
$$
\mbox{
$
\mathbf G_i
$ ($i\in\{1,\ldots,n\}$)
\ and \  
$\mathbf Z_{i}$ 
($i\in\{1,\ldots,n-1\}$)
are finitely generated.
}
$$
\qed\end{proof}

Next we show that two (and thus also finitely many) consecutive type II extensions can be replaced by a single type II extension.
More formally, we claim that 
\begin{proposition}\label{AB-C=A-BC}
For any odd FL$_e$-algebras $\mathbf A$, $\mathbf B$, $\mathbf C$, it holds true that
$$
\PLPII{(\PLPII{\mathbf A}{\mathbf B})}{\mathbf C}\simeq\PLPII{\mathbf A}{(\PLPII{\mathbf B}{\mathbf C})}
,
$$ 
that is, if the algebra on one side is well-defined then the algebra on the other side is well-defined, too, and the two algebras are isomorphic.
\end{proposition}
\begin{proof}
By Definition~\ref{FoKonstrukcio}, the universe of
$\PLPII{\mathbf A}{\mathbf B}$ is 
$(A_{gr}\times B)\cup(A\times\{\top\}),$
and hence the universe of
$(\PLPII{\mathbf A}{\mathbf B})_\mathbf{gr}$
is
$A_{gr}\times B_{gr}.$
Therefore, the universe of
$
\PLPII{(\PLPII{\mathbf A}{\mathbf B})}{\mathbf C}
$
is
$(A_{gr}\times B_{gr}\times C)\cup(A_{gr}\times B\times \{\top\})\cup(A\times\{\top\}\times\{\top\})$.
On the other hand, the universe of
$\PLPII{\mathbf B}{\mathbf C}$
is
$(B_{gr}\times C)\cup(B\times\{\top\}).$
Therefore, 
by denoting $(\top,\top)$ the new top element added to 
$\PLPII{\mathbf B}{\mathbf C}$,
the universe of 
$\PLPII{\mathbf A}{(\PLPII{\mathbf B}{\mathbf C})}
$
is
$(A_{gr}\times B_{gr}\times C)\cup(A_{gr}\times B\times\{\top\})\cup(A\times\{\top\}\times\{\top\}),$
so the underlying universes 
of $\PLPII{(\PLPII{\mathbf A}{\mathbf B})}{\mathbf C}$
and
$\PLPII{\mathbf A}{(\PLPII{\mathbf B}{\mathbf C})}$
coincide.
Clearly, the unit elements are the same. 
Since the monoidal operation of a type II extension is defined coordinatewise, the respective monoidal operations coincide, too.
Since both algebras are residuated and the monoidal operation uniquely determines its residual operation, it follows that the residual operations coincide, too, hence so do the residual complements.
\\
Finally, the left-hand side is well defined if and only if $A_{gr}$ is discretely embedded into $A$, and $A_{gr}\times B_{gr}$ is discretely embedded into $(A_{gr}\times B)\cup(A\times\{\top\})$. 
It cannot be the case that $|B_{gr}|=1$, since then for $a\in A_{gr}$, $(a,1_B)\in A_{gr}\times B_{gr}$ but $(a,1_B)_\uparrow=(a,\top)\notin A_{gr}\times B_{gr}$.
Therefore $|B_{gr}|=\infty$, and hence, using that $A_{gr}\times B_{gr}$ is discretely embedded into $(A_{gr}\times B)\cup(A\times\{\top\})$, it follows that $B_{gr}$ is discretely embedded into $B$.
Thus $\PLPII{\mathbf B}{\mathbf C}$ and also the right-hand side is well-defined, too.
On the other hand, the right-hand side is well-defined if and only if $A_{gr}$ is discretely embedded into $A$, and $B_{gr}$ is discretely embedded into $B$.
But then clearly $A_{gr}\times B_{gr}$ is discretely embedded into $(A_{gr}\times B)\cup(A\times\{\top\})$, and hence the left-hand side is well-defined, too.
\qed\end{proof}

Next, we define three series of odd FL$_e$-chains and state some basic properties of them.
\begin{definition}
For $j\in\mathbb N$, $j\geq 1$, let
$$
\begin{array}{lllllll}
\mathbf Z_0&:=&\mathbf Z_1&:=&\mathbb Z, \ \ \ \mathbf Z_{j+1}&:=&\PLPII{\mathbb Z}{\mathbf Z_j}, \\
\mathbf Q_0&:=&\mathbf Q_1&:=&\mathbb Q, \ \ \ \mathbf Q_{j+1}&:=&\PLPI{\mathbb Q}{\mathbb Z}{\mathbf Q_j}, \\
\mathbf R_0&:=&\mathbf R_1&:=&\mathbb R, \ \ \ \mathbf R_{j+1}&:=&\PLPI{\mathbb R}{\mathbb Z}{\mathbf R_j}.
\end{array}
$$
\end{definition}
\begin{proposition}\label{SpeciAlgebrak}
For $j\geq 1$, the following statements hold true.
\begin{enumerate}
\item
\begin{enumerate}
\item\label{charZj}
$Z_j=\{ (x_1,\ldots,x_j)\in\mathbb Z\times(\mathbb Z\cup\{\top\})\times\ldots\times(\mathbb Z\cup\{\top\}) \ |$ if $x_i=\top$ for some $i\in\{2,\ldots,j\}$ then for all $i\leq l\leq j$, $x_l=\top \}$,
\item\label{ZGRalaphalmaza}
$(\mathbf Z_j)_\mathbf{gr}=\underbrace{\mathbb Z\lex\ldots\lex \mathbb Z}_{j}$,
\item\label{ZjDISCemb}
$\mathbf Z_j$ is countable, 
$\mathbf Z_j$ has no least neither greatest element,
$\mathbf Z_j$ is Dedekind complete,
\item\label{ZjDISCRETELY}
$(\mathbf Z_j)_\mathbf{gr}$ is discretely embedded into $\mathbf Z_j$, and 
for any $x\in\mathbf Z_j\setminus (\mathbf Z_j)_\mathbf{gr}$ and $x\neq y\in\mathbf Z_j$, there exists $z\in(\mathbf Z_j)_\mathbf{gr}$ such that $z$ is strictly in between $x$ and $y$.
\item\label{noddogel}
$\PLPII{\mathbf Z_j}{\mathbf Z_k}\simeq \mathbf Z_{j+k}$.
\end{enumerate}
\item
\begin{enumerate}
\item\label{RGRigyNEZki}
$
(\mathbf R_j)_\mathbf{gr}=\underbrace{\mathbb Z\lex\ldots\lex \mathbb Z}_{j-1}\lex\mathbb R ,
$
\item\label{RiAZr}
$\mathbf R_{k_i}$ ($i\in\{1,\ldots,n\}$) is order isomorphic to $\mathbb R$.
\end{enumerate}
\end{enumerate}
\end{proposition}
\begin{proof}
By using that $Z_{j+1}$ (the universe of $\mathbf Z_{j+1}$) is equal to $\mathbb Z\times( Z_j \cup\{\top\})$, (\ref{charZj}) follow by an easy induction on $j$.
(\ref{ZGRalaphalmaza})-(\ref{noddogel}) readily follow from (\ref{charZj}), 
(\ref{noddogel}) follows from Proposition~\ref{AB-C=A-BC}, too.
By using that $R_{j+1}$ (the universe of $\mathbf R_{j+1}$) is $(\mathbb Z\times( R_j \cup\{\top\}))\cup(\mathbb R\times \{\bot\})$, (\ref{RGRigyNEZki}) follows by an easy induction on $j$.
An easy induction on $j$ also shows that 
$\mathbf R_j$ has no least neither greatest element, 
$\mathbf R_j$ is densely ordered,  
$Q_j$ is a countable dense subset of $\mathbf R_j$, and
$\mathbf R_j$ is Dedekind complete.
These, by Lemma~\ref{realisLEMMA}, conclude the proof of (\ref{RiAZr}).
\qed\end{proof}

\begin{proposition}\label{sdhajkdhasgGGG}
Let $\mathbf A$, $\mathbf B$ and $\mathbf L$ be odd FL$_e$-chains, $\mathbf H\leq\mathbf A_\mathbf{gr}$.
If 
\begin{itemize}
\item
$\mathbf A$ and $\mathbf B$ are order isomorphic to $\mathbb R$,
\item
$H$ and $L_{gr}$ are countable, 
\item
$L$ is Dedekind complete, has a countable dense subset, and has no least neither greatest element, 
\item
$L_{gr}$ is discretely embedded into $L$,
and 
\item
there exists no gap in $L$ formed by two elements of $L\setminus L_{gr}$
\end{itemize}
then
$$\mathbf D=\PLPII{(\PLPI{\mathbf A}{\mathbf H}{\mathbf L})}{\mathbf B}$$
is well-defined and order isomorphic to $\mathbb{R}$.
\end{proposition}
\begin{proof}
Denote $\mathbf C=\PLPI{\mathbf A}{\mathbf H}{\mathbf L}$.
Note that since $L$ is unbounded and $L_{gr}$ is discretely embedded into $L$, it follows that 
$C_{gr}= H\times L_{gr}$ is discretely embedded into $C=(H\times (L\cup\top_L))\cup(A\times\{\bot_L\})$.
Thus $\mathbf D$ is well-defined.
By definition, 
$$
D
=
(C_{gr}\times B)\cup(C\times \{\top_{B}\})
=
$$
$$
(H\times L_{gr}\times B)\cup(C\times \{\top_{B}\})
=
$$
$$
[H\times L_{gr}\times B]
\cup
[H\times (L\cup\{\top_L\})\times \{\top_{B}\}]
\cup
[A\times\{\bot_L\}\times \{\top_{B}\}]
=
$$
$$
[H\times L_{gr}\times B]
\cup
[H\times L\times \{\top_{B}\}]
\cup
[H\times \{\top_L\}\times \{\top_{B}\}]
\cup
[A\times\{\bot_L\}\times \{\top_{B}\}]
=
$$
$$
[H\times L_{gr}\times (B\cup \{\top_{B}\})]
\cup
[H\times (L\setminus L_{gr})\times \{\top_{B}\}]
\cup
[H\times \{\top_L\}\times \{\top_{B}\}]
\cup
[A\times\{\bot_L\}\times \{\top_{B}\}]
.
$$
$D$ has no least neither greatest element, since partial lexicographic products clearly inherit the boundedness of their first component, and
$A$ has neither least nor greatest element.

\medskip\noindent
We prove that $D$ is densely ordered.
Let $x=(x_1,x_2,x_3),(y_1,y_2,y_3)\in D$, $x<y$. 
\begin{itemize}
\item[$i.$]
Assume $x_1<y_1$.\\
Then there exists $z_1\in A$ such that $x_1<z_1<y_1$ since $\mathbf A$ is order isomorphic to $\mathbb{R}$, and hence 
$\mathbf A$ is densely ordered, hence $(z_1,\bot_L,\top_{B})$ is strictly in between $x$ and $y$.
\item[$ii.$]
Next, assume $x_1=y_1$ and $x_2<y_2$.\\
Since $x_2,y_2\in L\cup\{\bot_{L},\top_{L}\}$ and $x_2<y_2$ excludes $x_2=\top_{L}$, it follows that $x_2\in L\cup\{\bot_{L}\}$.
\begin{itemize}
\item
Assume $x_2=\bot_{L}$.\\
Then $y_2\in L\cup\{\top_{L}\}$ follows from $x_2<y_2$, and we can choose $c\in L$ such that $c<y_2$ since $L$ is nonempty and has no least element; resulting in $(x_1,c,\top_{B})$ being strictly in between $x$ and $y$.

\item
Assume $x_2\in L_{gr}$.\\
Then $(x_2)_\uparrow\in L_{gr}$ follows\footnote{$_\uparrow$ is computed in $L$.}  
since 
$L_{gr}$ is discretely embedded into $L$.
In addition, $x_2<(x_2)_\uparrow\leq y_2$ holds.
If $(x_2)_\uparrow<y_2$ then the element $(x_1,(x_2)_\uparrow,\top_{B})\in D$ is strictly in between $x$ and $y$, whereas 
if $(x_2)_\uparrow=y_2$ then $y_2\in L_{gr}$, too, and hence $y_3\in B\cup \{\top_{B}\}$ follows. 
Since $\mathbf B$ has no least element, we can choose $c\in B$ such that $c<y_3$. Hence $(x_1,(x_2)_\uparrow,c)$ is strictly in between $x$ and $y$.
\item
Assume $x_2\in L\setminus L_{gr}$.\\
Then $x_2<y_2$ excludes $y_2=\bot_{L}$, hence
$y_2\in L\cup \top_{L}$. 
\begin{itemize}
\item
 If $y_2=\top_{L}$ then, since $L$ has no greatest element, there is an element $c\in L$ such that $x_2<c <y_2$.
 \item
 If $y_2\in L_{gr}$ then since $L_{gr}$ is discretely embedded into $L$, it follows that
$L\setminus L_{gr}\ni x_2\not=(y_2)_\downarrow\in L_{gr}$, 
hence letting $c=(y_2)_\downarrow$, $x_2<c<y_2$ holds.
\item
 If $y_2\in L\setminus L_{gr}$ then there is an element $c\in L$ such that $x_2<c<y_2$ since there exists no gap in $L$ formed by two elements of $L\setminus L_{gr}$.
\end{itemize}
In all the three previous cases, the element $(x_1,c,\top_{B})$ is strictly in between $x$ and $y$, and we are done.
\end{itemize}
\item[$iii.$]
Finally, assume $x_1=y_1$, $x_2=y_2$ and $x_3<y_3$.\\
It follows that $x_1\in H$, $x_2\in L_{gr}$, and $x_3\in B\cup \{\top_{B}\}$.
Since $\mathbf B$ is order isomorphic to $\mathbb R$ (and hence 
$\mathbf B$ is densely ordered and has no greatest element), 
it follows that $B\cup\{\top_{B} \}$ is densely ordered, too, hence there is an element $c\in L$ such that $x_3<c<y_3$, and thus the element $(x_1,x_2,c)$ is strictly in between $x$ and $y$.
\end{itemize}

\medskip\noindent
Next we prove that $D$ is Dedekind complete.
Take any nonempty subset of $V\subseteq D$ which has an upper bound $(b_1,b_2,b_3)\in D$.
Then
$$
V_1=\{v_1 \ | \ (v_1,v_2,v_3)\in D\}
$$
is a subset of $A$, it is nonempty and bounded from above by $b_1$. Since $A$ is order isomorphic to $\mathbb{R}$, and thus
$A$ is Dedekind complete, there exists the supremum $m_1$ of $V_1$.
\begin{itemize}
\item[$i.$]
If $m_1\notin V_1$ then $(m_1,\bot_{L},\top_{B})\in D$ is the supremum of $V$. 
\item[$ii.$]
If $m_1\in V_1$ and $m_1\in A\setminus H$ then 
$(m_1,\bot_L,\top_{B})\in D$ is the supremum of $V$.
\item[$iii.$]
Finally, assume that $m_1\in V_1$ and $m_1\in H$.
Then
$$
V_2=\{v_2 \ | \ (m_1,v_2,v_3)\in V\}
$$
is a subset of $L \cup\{\top_L\}$, 
it is nonempty and bounded from above by $b_2$.
Since 
$L$ is Dedekind complete, 
so does $L\cup\{\top_{L}\}$, hence 
there exists the supremum $m_2$ of $V_2$ in $L \cup\{\top_L\}$.
\begin{itemize}
\item
If $m_2\in V_2$ and $m_2\in (L\cup{\top_L})\setminus L_{gr}$ then $(m_1,m_2,\top_{B})\in D$ is the supremum of $V$.
\item
If $m_2\in V_2$ and $m_2\in L_{gr}$ then
$$
V_3=\{v_3 \ | \ (m_1,m_2,v_3)\in V\}
$$
is a subset of $B \cup\{\top_{B}\}$, it is nonempty and bounded from above by $b_3$.
Since $B$ is order isomorphic to $\mathbb R$, and hence 
$B$ is Dedekind complete, there exists the supremum $m_3$ of $V_3$ in $B \cup\{\top_{B}\}$. Then $(m_1,m_2,m_3)\in D$, and  it is the supremum of $V$.
\item
If $m_2\notin V_2$ then it follows that $m_2$ cannot be an element of $L_{gr}$, since $L_{gr}$ is discretely embedded into $L$. 
Indeed, if $m_2\in L_{gr}$ were an upper bound of $V_2$, and $m_2\notin V_2$ then $(m_2)_\downarrow(<m_2)$ would be an upper bound of $V_2$, too, and hence $m_2$ cannot be the smallest upper bound.
Hence, $m_2\in (L\cup\{\top_L\})\setminus L_{gr}$, and thus $(m_1,m_2,\top_{B})\in D$ is the supremum of $V$.
\end{itemize}
\end{itemize}

\medskip\noindent
Finally, let 
$D$ and $D_4$ be countable dense subsets of $A$ and $B$, respectively.
Then,
since $H$ and $L_{gr}$ are countable, and $L$ has a countable dense subset $D_L$,
$$
[H\times L_{gr}\times D_4]
\cup
[H\times (D_L\cup\{\top_L\})\times \{\top_{B}\}]
\cup
[D\times\{\bot_L\}\times \{\top_{B}\}]
$$
is clearly a countable, dense subset of $D$.
\qed
\end{proof}

\begin{proposition}\label{topJOLESZ}
Let $\mathbf A$ and $\mathbf B$ be odd FL$_e$-chains which are order isomorphic to $\mathbb R$, $\mathbf H\leq\mathbf A_\mathbf{gr}$, $\mathbf H$ is countable, $\mathbf L=\mathbf Z_j$ for some $j\in\mathbb N$.
Then 
$\mathbf D=\PLPII{(\PLPI{\mathbf A}{\mathbf H}{\mathbf L})}{\mathbf B}$ is well-defined and order isomorphic to $\mathbb{R}$.
\end{proposition}
\begin{proof}
Proposition~\ref{SpeciAlgebrak} ensures all the properties of $\mathbf Z_j$ which are required in Proposition~\ref{sdhajkdhasgGGG}.
\qed\end{proof}

\begin{proposition}\label{topbotJOLESZ}
Let $\mathbf A$ and $\mathbf D$ be odd FL$_e$-chains which are order isomorphic to $\mathbb R$, $\mathbf H\leq\mathbf A_\mathbf{gr}$, $\mathbf H$ is countable. 
Let 
$\mathbf C=\PLPI{\mathbf A}{\mathbf H}{\mathbf D}$.
Then $\mathbf C$ is order isomorphic to $\mathbb{R}$.
\end{proposition}
\begin{proof}
By Definition~\ref{FoKonstrukcio}, the universe of $\mathbf C$ is
$$
C
=
(H\times(D\cup\{\top\}))\cup \left(A\times \{\bot\}\right)
=
(H\times(D\cup\{\top,\bot\}))\cup^{.} \left((A\setminus H)\times \{\bot\}\right)
.
$$
-
$C$ has no least neither greatest element since $A\times \{\bot\}\subseteq C$ and $A$ has no least neither greatest element.
\\
-
$C$ 
is densely ordered.
Indeed, let $(p,q)<(r,s)$.
If $p<r$ then there exists an element $v\in ]p,r[$ since $A$ is densely ordered, and thus $(v,\bot)\in C$ and $(p,q)<(v,\bot)<(r,s)$ holds.
If and $p=r$ then $p\in H$ follows and hence there is an element $v\in ]q,s[$ such that $(p,q)<(p,v)<(p,s)$ holds since $D\cup\{\top,\bot\}$ is densely ordered.
\\
-
Let $Q$ and $Q_3$ be countable dense subsets of $A$ and $D$, respectively.
It follows that $Q_2=(H\times(Q_3\cup\{\top\}))\cup \left(Q\times \{\bot\}\right)$ is a dense subset of $C=(H\times(D\cup\{\top\}))\cup \left(A\times \{\bot\}\right)$. Moreover, $Q_2$ is countable, since so is $H$.
\\
-
Finally we prove that $C$ is Dedekind complete.
Take any nonempty subset of $V\subseteq C$ which has an upper bound $(b_1,b_2)\in C$.
Let $V_1=\{v_1 \ | \ (v_1,v_2)\in C\}$.
Then $V_1\subseteq A$ is nonempty and bounded from above by $b_1$. Since $A$ is order isomorphic to $\mathbb{R}$, and $\mathbb{R}$ is Dedekind complete, there exists the supremum $m_1$ of $V_1$.
If $m_1\notin V_1$ then $(m_1,\bot)\in C$ is the supremum of $V$.
If $m_1\in V_1$ and $m_1\in A\setminus H$ then 
$(m_1,\bot)\in C$ is the supremum of $V$.
Finally, if $m_1\in V_1$ and $m_1\in H$ then $V_2:=\{v_2 \ | \ (m_1,v_2)\in V\}\subseteq D\cup\{\top,\bot\}$ is nonempty and bounded from above by $b_2$.
Since $D$ is Dedekind complete, so does $D\cup\{\top,\bot\}$, and hence $V_2$ has a supremum $m_2$ in $D\cup\{\top,\bot\}$, yielding that $(m_1,m_2)$ is the supremum of $V$.
\qed\end{proof}

\begin{theorem}\label{BeAgyazatam}
The logic $\mathbf{IUL}^{fp}$ enjoys finite strong standard completeness.
\end{theorem}
\begin{proof}
$\rm{IUL}^{fp}$-chains are non-trivial bounded odd FL$_e$-chains.
Knowing the fact that $\rm{IUL}^{fp}$-algebras constitute an algebraic semantics of $\mathbf{IUL}^{fp}$, we shall prove that any non-trivial finitely 
generated bounded odd FL$_e$-chain 
embeds into a odd FL$_e$-chain over the real unit interval $[0,1]$ such that its top element is mapped into $1$ and its bottom element is mapped into $0$.
Thus we prove that any $\mathbf{IUL}^{fp}$ formula which is falsified in a linearly ordered model of finitely many $\mathbf{IUL}^{fp}$ formulas (which is always a finitely generated $\rm{IUL}^{fp}$-chain)
can also be falsified in a standard odd FL$_e$-algebra, that is, in one over $[0,1]$.

To this end, 
let $\mathbf Y_1$ be a non-trivial finitely generated bounded odd FL$_e$-chain (thus it has at least three elements), and let $\mathbf Y$ be the subalgebra of $\mathbf Y_1$ over its universe deprived its top and bottom elements.
Then $\mathbf Y$ is a finitely generated (not necessarily bounded) odd FL$_e$-chain.

Our plan is to embed $\mathbf Y$, guided by its group representation, into a odd FL$_e$-chain $\mathbf X_n^\ast$ over a universe which is order isomorphic to $\mathbb{R}$. 
Then, using the order isomorphism together with an order isomorphism between $\mathbb{R}$ and $]0,1[$, we can carry over the structure of $\mathbf X_n^\ast$ into $]0,1[$, and finally we can add a top and a bottom element (as in item A in Definition~\ref{FoKonstrukcio}) 
to get a odd FL$_e$-chain over $[0,1]$, in which $\mathbf Y_1$ embeds. \footnote{We work with $\mathbf Y$ rather than with $\mathbf Y_1$ because we do not want to bother with the embedding of the top and bottom elements throughout each step of the proof.}

By Proposition~\ref{fingen_finIdemp}, $\mathbf Y$ has finitely many (say, $n\geq 1$) positive idempotent elements, therefore it has a type III-IV group representation by Theorem~\ref{Hahn_type}.
By Proposition~\ref{typeIIisVAN}, there exists an odd FL$_e$-chain $\mathbf X_n$ such that $\mathbf X_n$ has a type I-II group representation and
$$\mathbf Y\leq\mathbf X_n.$$
By Proposition~\ref{aGROUPREPRisFINITELYgeneratedTAGOKBOLall}, all the groups 
$\mathbf G_i$ $i\in\{1,\ldots,n\}$ in the group representation of $\mathbf Y$ (see (\ref{EzazY})) are finitely generated.
Hence, 
%
%
%
%
by the fundamental theorem of finitely generated abelian groups, each such group is isomorphic to a direct sum of cyclic groups, and since totally ordered groups are torsion fee, therefore totally ordered finitely generated abelian groups are isomorphic to direct sums of finitely many $\mathbb Z$'s.
So, for $i\in\{1,\ldots,n\}$, 
$
\mathbf G_i
\simeq
\bigoplus_{i=1}^{k_i}
\mathbb Z
$
holds for some $k_i\in\mathbb N$, 
where for $k_i=0$, $\mathbf G_i$ is meant to be the one-element group.
So $\mathbf G_i$ is countable, and unless $\mathbf G_i$ is the one-element group, $\mathbf G_i$ is discretely ordered.
The isomorphism naturally extends to an isomorphism between odd FL$_e$-chains.
Referring to claims~(\ref{ZGRalaphalmaza}) and (\ref{RGRigyNEZki}) in Proposition~\ref{SpeciAlgebrak}, for $i\in\{1,\ldots, n\}$, qua odd FL$_e$-chains,
\begin{equation}\label{follemehetunk}
\mathbf{G}_i\simeq \bigoplus_{i=1}^{k_i}\mathbb Z=(\mathbf Z_{k_i})_\mathbf{gr}
\leq(\mathbf R_{k_i})_\mathbf{gr}\leq\mathbf R_{k_i}.
\end{equation}

\medskip\noindent
Now, we are ready to define the embedding of 
$$
\mbox{
$\mathbf X_n$ into $\mathbf X_n^\ast$, where $\mathbf X_n^\ast$ is order isomorphic to $\mathbb R$,
}
$$
 as follows:
For $i\in\{1,\ldots,n\}$ let
\begin{equation}\label{slkalajksdhsjkjhkds}
\mathbf X_i^\ast=
\left\{
\begin{array}{ll}
\mathbf R_{k_1} & \mbox{if $i=1$ and either $\iota_2=III$ or $i=n$}\\
\mathbf Z_{k_1} & \mbox{if $i=1$, $\iota_2=IV$}\\
\PLPI{\mathbf X_{i-1}^\ast}{\mathbf Z_{i-1}}{\mathbf R_{k_i}}
& \mbox{if $i\in\{2,\ldots,n\}$, $\iota_i=III$, and either $\iota_{i+1}=III$ or $i=n$}\\
\PLPI{\mathbf X_{i-1}^\ast}{\mathbf Z_{i-1}}{\mathbf Z_{k_i}} 
& \mbox{if $i\in\{2,\ldots,n-1\}$, $\iota_i=III$, $\iota_{i+1}=IV$}\\
\PLPII{\mathbf X_{i-1}^\ast}{\mathbf Z_{k_i}}
& \mbox{if $i\in\{2,\ldots,n-1\}$, $\iota_i=IV$, $\iota_{i+1}=IV$}\\
\PLPII{\mathbf X_{i-1}^\ast}{\mathbf R_{k_i}}
& \mbox{if $i\in\{2,\ldots,n\}$, $\iota_i=IV$ and either $\iota_{i+1}=III$ or $i=n$}\\
\end{array}
\right. .
\end{equation}
We claim that for $i\in\{1,\ldots,n\}$, $\mathbf X_i$ embeds into $\mathbf X_i^\ast$, 
and that 
$\mathbf X_n^\ast$ is order isomorphic to $\mathbb R$;
thus concluding the proof of the theorem. 
\medskip
\\
(i) To see that $\mathbf X_i^\ast$ is well-defined, we need to prove the following two items.
\begin{itemize}
\item for $i\in\{2,\ldots,n\}$, $Z_{i-1}\subseteq (X_{i-1}^\ast)_{gr}$ holds.
Indeed,  
$Z_{1}\subseteq G_1 \subseteq (X_1^\ast)_{gr}$ holds by Proposition~\ref{typeIIisVAN}, and (\ref{slkalajksdhsjkjhkds}) and  (\ref{follemehetunk}).
By Proposition~\ref{typeIIisVAN} it holds true, for $3\leq i\leq n$, that $Z_{i-1}\subseteq V_{i-2}\lex G_{i-1}\subseteq Z_{i-2}\lex G_{i-1}$ which is a subset of $(X_{i-1}^\ast)_{gr}$ by (\ref{follemehetunk}) and (\ref{slkalajksdhsjkjhkds}).

\item If $i\in\{2,\ldots,n\}$, $\iota_i=IV$ then $(X_{i-1}^\ast)_{gr}$ is discretely embedded into $X_{i-1}^\ast$.
Indeed, if $\iota_2=IV$ then $\mathbf X_1^\ast=\mathbf Z_{k_1}$, $(Z_{k_1})_{gr}\simeq G_i$ holds by (\ref{follemehetunk}), and $G_i$ is discretely embedded into $Z_{k_1}$, as shown in claim~(\ref{ZjDISCRETELY}) of Proposition~\ref{SpeciAlgebrak}.
If, for $3\leq i\leq n$, $\iota_i=IV$ then by (\ref{slkalajksdhsjkjhkds}),
$\mathbf X_{i-1}^\ast$ is equal to either
$\PLPI{\mathbf X_{i-1}^\ast}{\mathbf Z_{i-1}}{\mathbf Z_{k_i}} $
or
$\PLPII{\mathbf X_{i-1}^\ast}{\mathbf Z_{k_i}}$. Hence $(X_{i-1}^\ast)_{gr}$,
which is equal to either 
$Z_{i-1}\lex { (Z_{k_i})_{gr}}$
or
$(X_{i-1}^\ast)_{gr}\lex { (Z_{k_i})_{gr}}$
is 
discretely embedded into the universe of either
$\PLPI{\mathbf X_{i-1}^\ast}{\mathbf Z_{i-1}}{\mathbf Z_{k_i}} $
or
$\PLPII{\mathbf X_{i-1}^\ast}{\mathbf Z_{k_i}}$, respectively,
since 
$(Z_{k_i})_{gr}$ is discretely embedded into $Z_{k_i}$.
\end{itemize}

\noindent
(ii) We prove that for $i\in\{1,\ldots,n\}$, $\mathbf X_i$ embeds into $\mathbf X_i^\ast$.
It holds for $i=1$ by (\ref{follemehetunk}).
Assume that for $i\in\{2,\ldots,n\}$, $\mathbf X_{i-1}$ embeds into $\mathbf X_{i-1}^\ast$.
If $\mathbf X_i^\ast$ is equal to either
$\PLPI{\mathbf X_{i-1}^\ast}{\mathbf Z_{i-1}}{\mathbf Z_{k_i}} $
or 
$\PLPI{\mathbf X_{i-1}^\ast}{\mathbf Z_{i-1}}{\mathbf R_{k_i}}$
(see the third and fourth rows of (\ref{slkalajksdhsjkjhkds}))
then 
$\mathbf X_i=\PLPII{\mathbf X_{i-1}}{\mathbf G_i}$
clearly embeds into $\mathbf X_i^\ast$.
If $\mathbf X_i^\ast$ is equal to 
$\PLPII{\mathbf X_{i-1}^\ast}{\mathbf R_{k_i}}$
(see the last row of (\ref{slkalajksdhsjkjhkds}))
then 
$\mathbf X_i=\PLPI{\mathbf X_{i-1}}{\mathbf Z_{i-1}}{\mathbf G_i}$
embeds into $\mathbf X_i^\ast$ 
by (\ref{follemehetunk}).

\medskip
\noindent
(iii) We prove that for $i\in\{2,\ldots,n\}$, $(X_{i-1}^\ast)_{gr}$ is countable if $\iota_i=IV$.
\\
Note that because of Proposition~\ref{AB-C=A-BC}, the fifth row of (\ref{slkalajksdhsjkjhkds}), and claim~(\ref{noddogel}) in Proposition~\ref{SpeciAlgebrak}, we may safely assume that there are no two consecutive type II extensions in (\ref{slkalajksdhsjkjhkds}),
\begin{equation}\label{NoTwoType2s}
\mbox{so the fifth row of (\ref{slkalajksdhsjkjhkds}) never applies.}
\end{equation}
Just like in point (i) above, $(X_1^\ast)_{gr}=G_1$ if $\iota_2=IV$ and hence it is countable since $\mathbf G_1$ is finitely generated, and if $3\leq i\leq n$ and $\iota_i=IV$ then
$(X_{i-1}^\ast)_{gr}$ is either 
$Z_{i-1}\lex { (Z_{k_i})_{gr}}=Z_{i-1}\lex G_i$ (which is countable, since $Z_{i-1}$ and $G_i$ are finitely generated)
or
$(X_{i-1}^\ast)_{gr}\lex { (Z_{k_i})_{gr}}=(X_{i-1}^\ast)_{gr}\lex G_i$.
In the latter case we proceed as follows: 
By (\ref{NoTwoType2s}), $\mathbf X_{i-1}^\ast$ is the result of a type I extension and is equal to
$\PLPI{\mathbf X_{i-2}^\ast}{\mathbf Z_{i-2}}{\mathbf Z_{k_{i-1}}}$
or it is equal to $\mathbf Z_{k_1}$.
Hence, $(X_{i-1}^\ast)_{gr}$ is either $Z_{i-2}\lex Z_{k_{i-1}}$ or $G_1$.
In both cases $(X_{i-1}^\ast)_{gr}$ is countable, since $\mathbf G_i$, $\mathbf Z_{i-2}$ are finitely generated, thus countable, and by claim~(\ref{ZjDISCemb}) in Proposition~\ref{SpeciAlgebrak}, $Z_{k_{i-1}}$ is countable.

\medskip
\noindent
(iv) Finally, we prove that for $m\in M=\{ i\in \{1,\ldots,n-1\} \ | \ \iota_{i+1}=III \}\cup\{n\}$,
$\mathbf X_m^\ast$ is order-isomorphic to $\mathbf R$.

Indeed, it is clear that for $i\in\{1,\ldots,n\}$, $\mathbf X_i^\ast$ has no least neither greatest element since neither $\mathbf R_{k_1}$ nor $\mathbf Z_{k_1}$ has least or greatest element, and this property is clearly inherited via type I and type II extensions.

Let $m$ be the least element of $M\not=\emptyset$.
Because of (\ref{NoTwoType2s}), either $m=1$ or $m=2$ holds.
If $m=1$ then $\iota_2=III$ and hence $\mathbf X_1^\ast=\mathbf R_{k_1}$, so the statement follows from claim~(\ref{RiAZr}) in Proposition~\ref{SpeciAlgebrak}.
If $m=2$ then $\iota_3=III$ and $\iota_2=IV$, hence
$\mathbf X_1^\ast=\mathbf Z_{k_1}$
and
$
\mathbf X_2^\ast
=\PLPII{\mathbf Z_{k_1}}{\mathbf R_{k_2}}
$, where
$(\mathbf Z_{k_1})_{gr}=\mathbf G_1$.
Its universe,
$W=(Z_{k_1}\times\{\top\})\cup(G_1\times R_{k_1})$, is clearly densely ordered, and $(Z_{k_1}\times\{\top\})\cup(G_1\times Q_{k_1})$ is a countable and dense subset of it.
Take any subset of $V\subseteq W$ which has an upper bound $(b_1,b_2)\in W$.
Let $V_1=\{v_1 \ | \ (v_1,v_2)\in V\}$.
Then $V_1\subseteq Z_{k_1}$ is nonempty and bounded from above by $b_1$. Since $Z_{k_1}$ is Dedekind complete, there exists the supremum $m_1$ of $V_1$.
If $m_1\notin G_1$ then $(m_1,\top)\in W$ is the supremum of $V$.
If $m_1\in G_1$ then $m_1\in V_1$ holds, too, since $G_1$ is discretely embedded into $Z_{k_1}$.
Then $V_2:=\{v_2 \ | \ (m_1,v_2)\in V\}\subseteq R_{k_1}\cup\{\top\}$ is nonempty and bounded from above by $b_2$.
Since $R_{k_1}$ is Dedekind complete, there exists the supremum $m_2$ of $V_2$ in $R_{k_1}\cup\{\top\}$, and clearly, $(m_1,m_2)$ is in $W$ and it is the supremum of $V$.
Summing up, $\mathbf X_2^\ast$ is Dedekind complete, so we are done with the basic step of the induction.

Induction hypothesis: Assume that for $m\in M\setminus\{n\}$, $\mathbf X_m^\ast$ is order-isomorphic to $\mathbf R$, and let $r=\min\{ i\in M \ | \ i>m\}$ be the next element of $M$. 
Because of (\ref{NoTwoType2s}), either $r=m+1$ or $r=m+2$ hold.
If $r=m+1$ then either $\iota_{m+2}=III$ or $m+2=n$.
Since $m\in M\setminus\{n\}$, therefore $\iota_{m+1}=III$ holds, and hence the third row of (\ref{slkalajksdhsjkjhkds}) applies, yielding
$
\mathbf X_r^\ast=
\PLPI{(\mathbf X_m^\ast)}{\mathbf H_{m,1}}{\mathbf R_{k_r}}
$.
An application of Proposition~\ref{topbotJOLESZ} yields that $\mathbf X_r^\ast$ is order isomorphic to $\mathbb R$.
If $r=m+2$ then $\iota_{m+1}=III$, $\iota_{m+2}=IV$, and either $\iota_{m+3}=III$ or $m+3=n$.
Therefore, by the fourth row of (\ref{slkalajksdhsjkjhkds}), 
$\mathbf X_{m+1}^\ast=\PLPI{(\mathbf X_m^\ast)}{\mathbf H_{m,1}}{\mathbf Z_{k_r}}$
and 
by the last row of (\ref{slkalajksdhsjkjhkds}),
$\mathbf X_{m+2}^\ast=
\PLPII{\mathbf X_{m+1}^\ast}{\mathbf R_{k_{m_2}}}
$;
thus Proposition~\ref{topJOLESZ} ensures that $\mathbf X_r^\ast$ is order isomorphic to $\mathbb R$.
\qed
\end{proof}

\bibliographystyle{amsplain}

\end{document}